\UseRawInputEncoding 
\documentclass[reqno,10pt]{amsart}
\usepackage{amssymb,amsmath,amsthm,amsfonts,color}

\usepackage{mathrsfs,dsfont, comment,mathscinet}

\usepackage{enumerate,esint}
\usepackage[left=2.8cm,right=2.8cm,top=2.8cm,bottom=2.8cm]{geometry}
\usepackage{mathtools}
\usepackage{graphicx}

\usepackage{epstopdf}
\usepackage[colorlinks=true, pdfstartview=FitV, linkcolor=blue, 
            citecolor=blue, urlcolor=blue]{hyperref}

\allowdisplaybreaks
\numberwithin{equation}{section}
\theoremstyle{plain}
\newtheorem{theorem}{Theorem}[section]
\newtheorem{proposition}[theorem]{Proposition}
\newtheorem{lemma}[theorem]{Lemma}

\theoremstyle{definition}
\newtheorem{definition}[theorem]{Definition}
\newtheorem{remark}[theorem]{Remark}

\usepackage{esint}
\def\R{{\mathbb R}}
\def\N{{\mathbb N}}
\def\p{{\varphi}}
\def\O{{\Omega}}
\def\Vol{{\operatorname{Vol}}}
\def\B{{\mathbb B}}
\def\n{{\nabla}}

\def\e{{\varepsilon}}
\def\uar{{u_{\alpha,\rho}}}
\def\duar{{\dot u_{\alpha,\rho}}}

\newcommand\UUU{\color{black}}

\newcommand\EEE{\color{black}}
\title[Spectral optimisation of inhomogeneous plates]{Spectral optimization of inhomogeneous plates} 
\author[E. Davoli]{Elisa Davoli}
\address[E. Davoli]{Institute of Analysis and Scientific Computing, TU Wien, Wiedner Hauptstra\ss e 8-10, 1040 Vienna,
Austria}
  \email{elisa.davoli@tuwien.ac.at}
\author[I. Mazari]{Idriss Mazari}
\address[I. Mazari]{Institute of Analysis and Scientific Computing, TU Wien, Wiedner Hauptstra\ss e 8-10, 1040 Vienna,
Austria}
  \email{idriss.mazari@tuwien.ac.at}
\author[U. Stefanelli]{Ulisse Stefanelli} \address[U. Stefanelli]{Faculty of Mathematics, University of Vienna, Oskar-Morgenstern-Platz 1, A-1090 Vienna, Austria.
Vienna Research Platform on Accelerating Photoreaction Discovery, University of Vienna, W\"ahringerstra\ss e 17, 1090 Wien, Austria.
Istituto di Matematica Applicata e Tecnologie Informatiche {\it E. Magenes}, via Ferrata 1, I-27100 Pavia, Italy} \email{ulisse.stefanelli@univie.ac.at}


\begin{document}

\maketitle

% REQUIRED
\begin{abstract}
This article is devoted to the study of spectral optimisation  for  inhomogeneous plates.  In particular, we  optimise the first eigenvalue of a vibrating plate with respect to  its thickness and/or density. Our result is threefold. First, we prove existence of an optimal thickness,  using fine tools \UUU hinging on \EEE topological properties of rearrangement classes.  Second, in the case \UUU of a circular plate, \EEE  we provide  a characterisation of  this  optimal thickness  by means of  Talenti inequalities. Finally, we prove a  stability  result when assuming that the thickness and \UUU the \EEE density of the plate are  linearly  related. This proof relies on  $H$-convergence tools applied to biharmonic operators. 
\end{abstract}

% REQUIRED

\section{Introduction}
%\subsection{Motivation and scope of the article}
%\paragraph{Scope of the article}
The study of  eigenmodes  optimisation  is central to the  theory of inhomogeneous elastic plates  and is of  great applicative relevance. A  vast  literature has been devoted to the analysis of spectral optimisation problems for biharmonic operators, modelling plates of varying density and thickness  under different settings  \cite{Anedda2010,Anedda,Berchio2020,Bucur2011,BuosoFreitas,Buoso2013,Buoso2015,Colasuonno2019,Kang2017,Kao2021}.  In addition,  several contributions  are devoted  to inverse problems arising in the study of such inhomogeneous plates \cite{Jadamba2015,Manservisi2000,Marinov2013}. In  the  latter context, the main objective is to  identify  some structural \UUU descriptors \EEE of the plate under consideration, such as its thickness or its bending stiffness, and the outlook on the problem is mostly computational. 

The goal of this article is to provide answers to several theoretical questions that, to the best of  our  knowledge, have not received a mathematical treatment  so far.  We focus on the optimization of thickness and/or density with respect to the first eigenvalue. For fixed density, we prove the existence of an optimal thickness. This calls for the implementation of a delicate argument, based on rearrangements. We then investigate the symmetry of the optimal solution in specific geometries, showing analogies with previously studied cases  \cite{Anedda2010,Anedda}.  Eventually, we prove   a stability result for  the case  in which the thickness and the density of the plate are  linearly  related.

%\paragraph{Informal statement of the problem}
 In order to make the discussion more precise, let  $\Omega$ be a bounded domain in $\mathbb{R}^2$ with $\mathscr C^2$ boundary, representing the reference  mid-surface  configuration of  a thin  plate \UUU at rest, \EEE and let  $D,g \in L^\infty(\O)$. The function $D$ describes the varying {\it thickness} of the plate.  Its lower bound is normalized  by assuming that $D\geq 1$,  where the inequality is  meant to hold almost everywhere in $\O$.  The function $g$ accounts for the {\it heterogeneity} of the plate. We are hence  led   to consider  the first eigenvalue associated with the natural vibration of the plate.  In variational terms,    this eigenvalue admits the following Rayleigh-quotient  representation 
\begin{equation}
  \tilde\Lambda(D,g)=\inf_{u\in W^{2,2}(\O)\cap W^{1,2}_0(\O),\, u\neq0}\frac{\int_\O D \left(\Delta u\right)^2}{\int_\O g u^2}.\label{eq:la0}\end{equation}
The associated eigenfunction $v_{D,g}$ satisfies the following  elliptic problem 
\begin{equation}\label{eq:eig-D}\begin{cases}\Delta \left(D \Delta v_{D,g}\right)=\tilde{\Lambda}(D,g) g v_{D,g}\text{ in }\O,
\\ v_{D,g}=\Delta v_{D,g}=0\text{ on }\partial \O.\end{cases}\end{equation}
The most general formulation of the optimisation problem under consideration,  covering  questions  from   \cite{Anedda2010,Anedda,Berchio2020,Colasuonno2019,Kang2017,Kao2021}, is the study of the qualitative properties of  solutions to the minimisation  problem 
\begin{equation}\inf_{D,g}\tilde{\Lambda}(D,g).\label{eq:ref}\end{equation}

 From the modeling viewpoint, the reference \UUU application consists in \EEE reinforcing the plate locally by adding a layer of material, hence increasing the thickness, or by combining two materials, hence increasing the density. These cases \UUU translate in \EEE the choice 
% In other words, \emph{which are the optimal thickness and density from the stability point of view}? Of course, we need to specify our constraints on $D$ and $g$. Since, from a modelling point of view, it is natural to think that we are trying to reinforce the plate locally,  we may think of $D$ and $g$ as writing
\begin{equation}D=1+\beta_0 \mathds 1_\omega,\, g=1+\delta_0\mathds 1_{\omega'}\end{equation} for two measurable subsets $\omega,\, \omega'$ on which we can act,  where $\mathds 1$ is the corresponding characteristic function.  This in turn leads to considering $L^\infty$ and $L^1$ constraints on $D$ and $g$.
 We hence  introduce the following admissible classes  for thickness and  heterogeneity,  where $\beta_0,\, \delta_0,\, D_0,\, g_0$ are fixed  positive  parameters:  % numbers: 
 \begin{align}\mathcal N(\O):=&\left\{D\in L^\infty(\O) \ : \   1\leq D\leq 1+\beta_0,\, \int_\O D=D_0\right\},\label{eq:N}\\
   \mathcal N'(\O):=&\left\{g\in L^\infty(\O) \ : \ 1\leq g\leq 1+\delta_0,\, \int_\O g=g_0\right\}.\end{align} The  main minimization problem \eqref{eq:ref} is then specified as follows  %hen writes 
\begin{equation}\label{Eq:PvIntro}
\inf_{D\in \mathcal N(\O),\, g\in \mathcal N'(\O)}\tilde\Lambda(D,g).
\end{equation}

 Let us start by removing a 
 difficulty  related  to the definition of the eigenvalue, which is that the potential term $\tilde{\Lambda}(D,g)gv_{D,g}$  in \eqref{eq:eig-D}    appears in a  multiplicative form.  As  it  is customary in eigenvalue optimisation, arguing as in \cite[Theorem 13]{Chanillo2000}  we reformulate the problem by referring to the {\it density \UUU (excess)} $\rho$ of the plate instead of its heterogeneity. In particular, we  introduce the class  of admissible densities 
\begin{equation}\mathcal  M(\O):=\left\{\rho \in L^\infty(\O) \ : \  0\leq \rho\leq 1,\, \int_\O \rho=\rho_0\right\}\label{eq:M}\end{equation} and, for $D\in \mathcal N(\O)$,  we  define the first eigenvalue 
\begin{equation}\Lambda(D,\rho):=\inf_{u\in W^{2,2}(\O)\cap W^{1,2}_0(\O),\,  u \not = 0}\frac{\int_\O D \left(\Delta u\right)^2-\int_\O \rho u^2}{\int_\O u^2}.\label{eq:la}\end{equation} 
Up to a scaling factor, proceeding along the lines of \cite[Theorem 13]{Chanillo2000},  solving \eqref{Eq:PvIntro} is equivalent to  finding  solutions to  \begin{equation}\label{Eq:PvIntro2}
\inf_{D\in \mathcal N(\O),\,\rho \in \mathcal M (\O)}\Lambda(D,\rho).
\end{equation}
 We prefer to work with formulation \eqref{Eq:PvIntro2}, for the normalization term $\int_\O u^2=1$ in the denominator in \eqref{eq:la} is independent of $\rho$ (compare with \eqref{eq:la0}). 
% Given that, in  formulation \eqref{Eq:PvIntro2}, the normalisation condition $\int_\O u^2=1$ does not depend on $\rho$, we prefer to work with this version of the optimisation problem. 

 Most  contributions  on the minimization problem \eqref{Eq:PvIntro2} focus on the case of fixed thickness  $D\equiv 1$ and the optimisation is carried  out  with respect to $\rho$  only,   either  under Navier boundary conditions, or under clamped boundary conditions, see for instance \cite{Anedda2010,Anedda,Kang2017}. In these contributions, rearrangements arguments and Talenti inequalities are used in order to derive Faber-Krahn-like inequalities,  delivering information on the geometry of minimizers.  On the other hand, the optimisation of the thickness  $D$  is  mostly treated  numerically   \cite{Berchio2020,Jadamba2015,Manservisi2000,Marinov2013} and the existence of  a minimizer $D^*$  is usually not ascertained,  to the best of our knowledge. Let us stress that existence  in this setting can be quite delicate  to obtain.  As a matter of comparison, let us recall that in the somehow related case of
 optimisation of the first eigenvalue of two-phases operators $-\nabla \cdot(D\n)$ under the constraint $D\in \mathcal N(\O)$, it is well-known \cite{MuratTartar,CasadoDiaz3} that no solution exists if $\O$ is not a ball.

%\paragraph{Informal statement of the results}
 The first main result of the paper is hence an existence proof for an optimal thickness $D^*$ for \eqref{Eq:PvIntro2}, under fixed  $\rho\equiv 0$. In particular,
setting $\mu(D):=\Lambda(D,0)$, we investigate the minimisation problem $\underset{D\in \mathcal N(\O)}\inf \,\mu(D)$. We prove  in Theorem \ref{Th:ExistMu}  that, in any domain $\O$, a minimiser exists.  Note  that it is in sharp contrast with several other models involving heterogeneity in the leading term of the underlying elliptic operator (such as classical two-phases operators),  where existence strongly depends on the choice of the ambient space $\O$.  The proof of Theorem \ref{Th:ExistMu} relies on   delicate  topological properties of constraint classes defined through rearrangements and we will make use of some related results from  \cite{Alvino1989,ConcaMahadevanSanz}.

 Our second main result, Theorem \ref{Th:PV0},  focuses on the case when $\O$ is a ball.  In this case, we are able to  characterise the optimal thickness $D^*$  as being piecewise constant and radially symmetric. The argument \UUU is in \EEE the spirit   of \cite{Anedda,Anedda2010}.  In particular,  we use  Talenti inequalities  in combination with  rearrangement arguments.

 In our last main result, Theorem \ref{Th:Stability}, we  investigate the case  of  coupled  thickness  and density.  For simplicity, we focus on the case of a linear relation between these two quantities, namely,  $D=1+\alpha \rho$ for a small parameter $\alpha>0$.  Albeit linear, this case 
  already proves  very   challenging.  By  defining $\lambda_\alpha(\rho):=\Lambda(1+\alpha \rho,\rho)$, we give a fine stability analysis in the case where $\O$ is a ball, $\alpha$ is small, and all the functions involved are assumed to be radial.  In particular we obtain a \UUU stationary result: \EEE the minimisers $\rho^*$ in the case $\alpha=0$, which were \UUU already \EEE studied in  \cite{Anedda2010,Anedda,Kang2017}, remain optimal for $\alpha>0$ small enough. This proof relies on $H$-convergence-like tools, generalising to biharmonic operators a strategy developed in \cite{MazariNadinPrivat}.   

 The paper is organised as follows. In Section \ref{sec:setting} we specify the precise assumptions for our analysis and state our three main results. In Section \ref{Se:Preliminary} we collect some preliminary technical results. Sections \ref{sec:thm1}--\ref{sec:thm3} are devoted to the proofs of Theorems \ref{Th:ExistMu}--\ref{Th:Stability}. Eventually, Section \ref{sec:con} contains a summary of our findings.

\section{Mathematical setting and results}
 Throughout the paper, inequalities will always be meant in the sense of \UUU $L^1$ \EEE functions, namely almost everywhere in the corresponding set where the different quantities are defined. 
\label{sec:setting}
\subsection{Optimisation with respect to the thickness}
We first investigate optimisation with respect to the thickness of the plate.  Given two positive  parameters $\beta_0,\, D_0$, the admissible class of thicknesses  $\mathcal N(\O)$ is defined in \eqref{eq:N}, 
%\begin{equation}\mathcal N(\O)=\left\{D\in L^\infty(\O),\, 1\leq D\leq 1+\beta_0,\, \int_\O D=D_0\right\},
%\end{equation}
where  nonetheless  we assume that $D_0>\operatorname{Vol}(\O)$ in order to ensure that this class is not  empty or reduced to a single element.  For any $D\in \mathcal N(\O)$  we define  the first eigenvalue $\mu(D)$  given by  the Rayleigh quotient 
\begin{equation}\label{Eq:DefMu}\mu(D)=\inf_{u\in W^{1,2}_0(\O)\cap W^{2,2}(\O),\, u\neq 0}\frac{\int_\O D\left(\Delta u\right)^2}{\int_\O u^2},\end{equation} which is associated with the following eigenequation (where we have chosen a $L^2$ normalisation):
\begin{equation}\label{Eq:MuD}\begin{cases}
\Delta(D\Delta u_D)=\mu(D) u_D\text{ in }\O, 
\\ u_D=\Delta u_D=0\text{ on }\partial \O, 
\\  \int_\O u_D^2=1.
\end{cases}\end{equation}
We emphasise once again that this corresponds to problem \eqref{Eq:PvIntro} with $g\equiv 1$. The first optimisation problem we consider is 
\begin{equation}\label{Eq:PvMu}\inf_{D\in \mathcal N(\O)}\mu(D).\end{equation}
Our first result establishes the existence of a solution:
\begin{theorem}[Existence of minimisers]\label{Th:ExistMu}
For any bounded  domain $\Omega\subset \mathbb{R}^2$ with  $\mathscr C^2$  boundary  there exists $D^*\in \mathcal N(\O)$ such that 
\begin{equation}\inf_{D \in \mathcal N(\O)} \mu(D)=\mu(D^*).\end{equation}Furthermore, there exists a measurable set $\omega^*\subset \O$ such that $D^*=1+\beta_0 \mathds 1_{\omega^*}$. \end{theorem}
The proof of this theorem relies on rather fine topological arguments which  yield  compactness of sequences of minimisers. Let us note that, as is classical in  this class of  problems, one can not use the direct method in the calculus of variations: indeed, the best convergence one could get on a minimising sequence $\{D_k\}_{k\in \N}$ (and on the associated sequence of eigenfunctions $\{u_k\}_{k\in \N}$) is  the weak-$\ast$  convergence in $L^\infty $ of  $\{D_k\}_{k\in \N}$ and weak convergence of $\{u_k\}_{k\in \N}$ in $W^{2,2}(\O)$, thus forbidding to pass to the limit in the Rayleigh-quotient formulation \eqref{Eq:DefMu} of $\{\mu(D_k)\}_{k\in \N}$. This is  a known \UUU conundrum \EEE  in the study of two-phases operators \cite{MuratTartar},  \UUU impairing the proof of \EEE the existence of optimisers  for general $\O$.   We present here  a way to circumvent this difficulty in the case of biharmonic operators.

In general domains, it is hopeless to give an explicit characterisation of the optimal thickness $D^*$. In the case of a ball, however, using Talenti inequalities, we obtain  an inequality of Faber-Krahn type.  Consider the case in which our plate \UUU coincides \EEE with the ball of radius $R>0$ centered in the origin, i.e.  $\O=\mathbb B(0;R)$. Define the  function  ${\overline D}_\#$ as follows: 
\begin{equation}\label{Eq:DefA}
  \overline D_\#=(1+\beta_0)\mathds 1_{\mathbb A}+\mathds 1_{\mathbb A^c}
\end{equation}
where, in radial coordinates, $\UUU {\mathbb A }\EEE =\{r_0<r<R\}$  and  $ \operatorname{Vol}(\mathbb A)=({D_0-\operatorname{Vol}(\O)})/{\beta_0}$.  Note that the set 
$\mathbb A$ is uniquely defined  and the volume constraint ensures that $\overline D_\#\in \mathcal N(\O)$.  Our second result reads as follows. 

\begin{theorem}[The case of the ball]\label{Th:PV0}
Let $\O=\mathbb B(0;R)$  for some $R>0$. Then,  $\overline D_\#$ minimises $\mu$ in $\mathcal N(\O)$,  namely, 
\begin{equation}   \mu({\overline D}_\#) \leq \mu(D)  \quad \forall D \in \mathcal N(\O).\end{equation}\end{theorem}

It should be noted that this is the exact opposite  result with respect to  the optimisation  on  the density $\rho$ (i.e. keeping $D\equiv 1$).  In fact, by minimizing w.r.t. $\rho$  it is shown \UUU in \EEE \cite{Anedda2010,Anedda} that the unique optimal material density $\rho^*$ when $\O=\mathbb B(0;R)$ corresponds to having a maximal density in the center, and a minimal density close to the boundary: $\rho^*=\mathds 1_{\mathbb B(0;r^*)}$ with $r^*$ chosen so as to satisfy the volume constraint.  This \UUU observation  motivates \EEE our interest in investigating optimality with respect to density {\it and} thickness. We tackle this topic in the next subsection, by assuming a linear relation between $\rho$ and $D$.  % This is the topic of the next subsection.

\subsection{Density-dependent thickness}
In this subsection, we consider another version of \eqref{Eq:PvIntro}-\eqref{Eq:PvIntro2}, by assuming a linear dependency of the thickness $D$ of the plate with respect to the density \UUU of the material. \EEE In other words, we consider a real parameter $\alpha\geq 0$, and assume that the thickness $D$ depends on the density \UUU of the material \EEE via the relation
\begin{equation}
D=1+\alpha \rho.\end{equation}
Keeping in mind that $\rho$ corresponds to the repartition of some material inside the elastic plate $\O$, we  recall the  admissible class $\mathcal M(\Omega)$ of densities from \eqref{eq:M} 
% \begin{equation}\mathcal M(\O):=\left\{ \rho \in L^\infty(\O),\, 0\leq \rho\leq 1,\, \int_\O \rho=\rho_0\right\},\end{equation}
and, for any $\rho \in \mathcal M(\O)$, we consider the first eigenvalue $\lambda_\alpha(\rho)$ of \UUU $u \mapsto \Delta\Big((1+\alpha \rho)\Delta u\Big)-\rho u$. In its Rayleigh-quotient formulation, this is given\EEE by 
\begin{equation}\label{Eq:DefLambda}\lambda_\alpha(\rho):=\inf_{u\in W^{2,2}(\O)\cap W^{1,2}_0(\O),\, u\neq 0}\frac{\int_\O (1+\alpha \rho)(\Delta u)^2-\int_\O \rho u^2}{\int_\O u^2}.\end{equation} Up to a $L^2$ normalisation, the associated eigenfunction $\uar$ satisfies
\begin{equation}\label{Eq:LambdaRho}
\begin{cases}
\Delta \left((1+\alpha \rho)\Delta \uar\right)=\lambda_\alpha(\rho)\uar+\rho\uar\text{ in }\O, 
\\ \uar=\Delta \uar=0\text{ on }\partial \O, 
\\ \int_\O \uar^2=1.
\end{cases}
\end{equation}
We prove in Lemma \ref{Cl:Misc} that $\lambda_\alpha(\rho)$ is a simple eigenvalue and that the associated first eigenfunction has a constant sign.

 For a fixed   parameter $\alpha\geq 0$,  we consider  the optimisation problem
\begin{equation}\inf_{\rho \in \mathcal M(\O)}\lambda_\alpha(\rho).\end{equation} We assume that $\O=\mathbb B(0;R)$  for some $R>0$ \UUU and \EEE focus on the \UUU geometry of minimizers for \EEE $\alpha>0$ \UUU small. \EEE Indeed, an explicit characterisation of the minimisers for $\alpha=0$ was given in \cite{Anedda}: if $\B^*$ is the unique  ball centered in the origin, contained in $\O = \B(0;R)$ \UUU with \EEE $\operatorname{Vol}(\B^*)=\rho_0$, then the unique minimiser of $\lambda_0$ in $\mathcal M(\O)$ is 
\begin{equation}\rho^*=\mathds 1_{\B^*}.\end{equation} \UUU On the other hand, \EEE Theorem \ref{Th:PV0} seems to indicate that, for $\alpha\to \infty$, the optimal $\rho$ should behave as $\mathds 1_{\mathbb A}$, where $\mathbb A=\{r_0<r<R\}$ is the only  annulus  of volume $\rho_0$. %We refer to the conclusion for a discussion of this regime $\alpha \to \infty$.

\begin{theorem}[Stability for small $\alpha$ in the ball for radially symmetric distributions]\label{Th:Stability} Let $\O=\mathbb B(0;R)$  for some $R>0$,  and define $\rho^*:=\mathds 1_{\B^*}$.  Then, there  exists $\overline \alpha>0$ such that, for any $0\leq \alpha\leq \overline \alpha$, 
\begin{equation}
\lambda_\alpha(\rho^*)\leq \lambda_\alpha(\rho )  \quad \forall \rho \in \mathcal M(\B), \ \rho \text{ radially symmetric}. 
\end{equation}
\end{theorem}
The proof of this theorem relies on fine arguments inspired from $H$-convergence theory \cite{Allaire,MuratTartar}, and can be linked to some stationarity results in two-phases problems \cite{Laurain,MazariNadinPrivat}. In the proof, the radial symmetry assumption of competitors is crucial. %We refer to the conclusion for possible generalisation to non-radial competitors. %In the conclusion, we discuss a strategy to overcome this radiality assumption; however, we believe such a proof would be very lengthy, and we choose to omit it in its full form for the sake of brevity.

\section{Preliminary technical results}\label{Se:Preliminary}
We first gather in this section several preliminary results that are used throughout the rest of the paper.

 Let us begin by presenting a straightforward  application of the maximum principle.
\begin{lemma}[Positivity principle]\label{Cl:Positivity}
Let $\rho\in \mathcal M(\O)$.
Assume that $u\in W^{2,2}_0(\O)$ satisfies, for some $f\in L^2(\O)$,
\begin{equation}\begin{cases}
\Delta \left((1+\alpha \rho)\Delta u\right)=f\geq 0\text{ in }\O, 
\\ u=\Delta u=0\text{ on }\partial \O.\end{cases}\end{equation}
Then 
\begin{equation}
u\geq 0 \ \text{and} \ \  (1+\alpha \rho)\Delta u\leq 0\text{ in }\O.\end{equation}

\end{lemma}
\begin{proof}[Proof of  Lemma \ref{Cl:Positivity}]
Let $\rho,u$ be as in the statement of the lemma. First of all, by elliptic regularity, 
$$(1+\alpha \rho)\Delta u\in \UUU W^{2,2}\EEE(\O).$$ Let us introduce  $z=-(1+\alpha \rho)\Delta u$.  % the solution $z$ of $$-\Delta u=\frac{z}{1+\alpha \rho}.$$
Then $z\in W^{1,2}_0(\O)$ satisfies
\begin{equation*}
\begin{cases}
-\Delta z=f\geq 0\text{ in }\O,
 \\z=0\text{ on }\partial \O.
 \end{cases}
 \end{equation*}
As a consequence of the maximum principle for the Laplacian we obtain 
$z\geq 0\text{ in }\O.$ Hence, 
$
\Delta u\leq 0.$
Since $u\in W^{2,2}(\O)$, $\Delta u \in L^2(\O)$. We can then apply the maximum principle to the inequality
$-\Delta u\geq 0\text{ in }\O$ to conclude that $u\geq 0$ in $\O$. 
\end{proof}

 In the next lemma we collect  some basic facts about the underlying spectral and optimisation problems.
\begin{lemma}\label{Cl:Misc}
\begin{enumerate}
\item  For any $D\in \mathcal N(\O)$, $ \alpha\geq 0$, and $ \rho\in \mathcal M(\O)$, the eigenfunctions $u_D$ and $\uar$ can be assumed to have constant sign. Hence, the first eigenvalue is the only one  whose eigenfunction is constant in sign.   
\item There exists $M>0$ such that, for any $D\in \mathcal N(\O)$,
\begin{equation}\left|�\mu(D)\right|,\, \Vert u_D\Vert_{W^{2,2}(\O)}\leq M.\end{equation}
\item For any $\overline \alpha>0$, there exists $M(\overline \alpha)$ such that, for any $\rho \in \mathcal M(\O)$ and any $\alpha \in [0;\overline \alpha]$, 
\begin{equation}\left|\lambda_\alpha(\rho)\right|,\, \Vert \uar\Vert_{W^{2,2}(\O)}\leq M(\overline \alpha).\end{equation}
\end{enumerate}
\end{lemma}
\begin{proof}[Proof of Lemma \ref{Cl:Misc}]
To prove \UUU point 1 of the \EEE Lemma, we adapt \cite[Lemma 16]{Berchio2006}. We \UUU detail this argument for  $\lambda_\alpha(\rho)$ only, for an analogous proof yields \EEE the conclusion for $\mu(D)$, \UUU as well. \EEE  In order  to prove \UUU point 1, \EEE  it suffices  to establish the following fact: for any $u\in W^{2,2}(\O)\cap W^{1,2}_0(\O)$ and any $\rho \in \mathcal M(\O)$ there exists $w\in W^{2,2}(\O)\cap W^{1,2}(\O)$ such that
\begin{equation}w\geq 0, \   \int_\O (1+\alpha \rho)(\Delta w)^2-\rho w^2\leq  \int_\O (1+\alpha \rho)(\Delta u)^2-\rho u^2, \ \text{ and} \ \int_\O w^2\geq \int_\O u^2. \end{equation} Indeed, since $u\in W^{2,2}(\O)$ does not imply $|u|\in W^{2,2}(\O)$, it is not possible to simply replace $u$  by  its absolute value. Let us hence consider $\rho\in \mathcal M(\O),\, \alpha\geq 0,\, u\in W^{2,2}(\O)\cap W^{1,2}_0(\O)$ and define $w$ as the unique solution of 
\begin{equation}\begin{cases}
-\Delta w=\left|�\Delta u\right|\text{ in }\O, 
\\ w=0\text{ on }\partial \O.\end{cases}
\end{equation}
\UUU We \EEE first observe that 
$\int_\O (1+\alpha \rho)(\Delta w)^2=\int_\O (1+\alpha \rho)(\Delta u)^2.$
Besides, by the maximum principle, 
$w\geq 0\text{ in }\O.$ Furthermore, from the definition of $w$, we get that 
$-\Delta w\geq -\Delta u,\, -\Delta w\geq \Delta u$, whence $w\geq u$, and $w\geq -u$ in $\O$. As a consequence, $w\geq |u|$ in $\O$. Thus, $\int_\O w^2\geq \int_\O u^2.$
Since $\rho \geq 0$, \UUU we have that \EEE
$$\int_\O \rho u^2\leq \int_\O \rho w^2$$ which yields the conclusion. It should be noted that this construction proves that any eigenfunction associated with the first eigenvalue has a constant sign, whence the simplicity of the first eigenvalues $\mu(D)$ and $\lambda_\alpha(\rho)$.

We \UUU now proceed with the proof of point 2. Point 3 follows \EEE from the exact same arguments. Let us consider $D\in \mathcal N(\O)$. From the Rayleigh-quotient formulation \eqref{Eq:DefMu} of $\mu(D)$, we get that $\mu(D)\geq 0$ (for $\lambda_\alpha(\rho)$, we would get $\lambda_\alpha(\rho)\geq -1$). Let us consider the first eigenvalue $\eta_1(\O)$ of the  biharmonic operator in $\O$ defined as 
\begin{equation}\eta_1(\O):=\inf_{u\in W^{2,2}(\O)\cap W^{1,2}_0(\O),\, \int_\O u^2=1}\int_\O \left(\Delta u\right)^2.\end{equation}  Let $w_1$ be an associated eigenfunction. Then

\begin{equation}
\mu(D)\leq \int_\O D\left(\Delta w_1\right)^2\leq (1+\beta_0)\int_\O \left(\Delta w_1\right)^2=(1+\beta_0)\eta_1(\O),
\end{equation}
which yields the required uniform bound on the eigenvalue. Next,  by  multiplying the eigenequation \eqref{Eq:MuD} by $u_D$ and integrating by parts we obtain 
\begin{equation*}
\int_\O \left(\Delta u_D\right)^2\leq \int_\O D\left(\Delta u_D\right)^2=\mu(D)\leq (1+\beta_0)\eta_1(\O).
\end{equation*}
Since, by elliptic regularity, for any $u\in W^{1,2}_0(\O)$,
\begin{equation}
\Vert u_D\Vert_{W^{2,2}(\O)}\leq C\Vert \Delta u_D\Vert_{L^2(\O)}\end{equation} we obtain the required bound.
\end{proof}
Henceforth,  with no loss of generality we assume $u_D$ and $u_{\alpha,\rho}$ to be nonnegative,  % we chose,
up to multiplying them by \UUU $-1$. \EEE % , that $u_D$ and $\uar$ are non-negative.
 Our next step is hence to establish the concavity of  %We then derive a crucial property of
the eigenvalue maps. % $\rho\mapsto \lambda_\alpha(\rho)$, $D\mapsto \mu(D)$.
\begin{lemma}\label{Le:Concav}
Let $\alpha\geq 0$ be fixed. The two maps 
\begin{equation}
\mathcal N(\O)\ni D\mapsto \mu(D),\ \mathcal M(\O)\ni \rho\mapsto \lambda_\alpha(\rho)
\end{equation}
are concave.
\end{lemma}
\begin{proof}[Proof of Lemma \ref{Le:Concav}]
Each of these two maps is defined as an infimum of linear functionals in their respective variables, so that they are concave.
\end{proof}

This concavity property enables one to write the seemingly naive but in fact crucial reformulation of the eigenvalue problems in terms of bang-bang  functions,  which we now define.
\begin{definition}
A function $D\in \mathcal N(\O)$ is called bang-bang if $D=1+\beta_0\mathds 1_\omega$ for some measurable subset $\omega$ of $\O$. Such functions are the extremal points of $\mathcal N(\O)$ and are denoted $\operatorname{Ext}(\mathcal N(\O))$.

 A function $\rho \in \mathcal M(\O)$ is called bang-bang if $\rho=\mathds 1_{\omega'}$ for some measurable subset $\omega'$ of $\O$. Such functions are the extremal points of $\mathcal M(\O)$ and are denoted $\operatorname{Ext}(\mathcal M(\O))$.
\end{definition}
The definition of bang-bang functions in terms of extremal points is classical \cite[Proposition 7.2.17]{HenrotPierre}. As an immediate consequence of Lemma \ref{Le:Concav} and of the convexity of the admissible sets $\mathcal M(\O)$ and $\mathcal N(\O)$, we obtain the following lemma:

\begin{lemma}\label{Cl:Reformulation} We have \UUU that 
\begin{align*}
&\inf_{D\in \mathcal N(\O)}\mu(D)=\inf_{D\in \operatorname{Ext}(\mathcal N(\O))}\mu(D),\\
&\inf_{\rho\in \mathcal M(\O)}\lambda_\alpha(\rho)=\inf_{\rho\in \operatorname{Ext}(\mathcal M(\O))}\lambda_\alpha(\rho).
\end{align*}
\end{lemma}

\section{Proof of Theorem \ref{Th:ExistMu}}
\label{sec:thm1}
The proof relies on several preliminary results that we recall in \UUU Subsection \ref{sub:re}. \EEE The proof is \UUU then presented in Subsection \ref{sec:add}. \EEE 

\subsection{Preliminary material about rearrangements}
\label{sub:re}
Let us briefly recall the key concepts of \UUU the Schwarz \EEE rearrangement. For a comprehensive introduction to rearrangements, we refer to \cite{Bandle,Kawohl,Kesavan}.  For a $\mathscr C^2$  domain of $\R^2$,  let $\O^\#=\mathbb B(0; R^\#)$ be the  centered  ball with the same volume as $\O$. For any function $\p\in L^2(\O),\, \p\geq 0$, the Schwarz rearrangement of $\p$ is the unique non-increasing function $\p^\#:\O^\#\to \R_+$ such that, for any $t\geq 0,$
\begin{equation}\operatorname{Vol}\left(\{\p>t\}\right)=\operatorname{Vol}\left(\{\p^\#>t\}\right).\end{equation}
Of particular importance are the following properties of this rearrangement:
\begin{enumerate}
\item Equimeasurability: for any $\p \in L^2(\O)$, $ \p\geq 0$, $$\Vert \p\Vert_{L^2(\O)}=\Vert\p^\#\Vert_{L^2(\O^\#)}^2.$$
\item Hardy-Littlewood inequality: for any non-negative functions $\p_0,\p_1\in L^2(\O)$,
$$\int_\O \p_0\p_1\leq \int_{\O^\#}\p_0^\#\p_1^\#.$$
\end{enumerate}
Another key tool is \UUU the Talenti inequality \cite{Talenti}  which reads as follows. \EEE %: if $f\in L^2(\O),\, f\geq 0$ and $w$ solves
\begin{proposition}[Talenti inequality, {\cite[Theorem 1]{Talenti}}]\label{Pr:Talenti} Let $\O$ be a Lipschitz bounded domain, and let $\B$ be the ball centered in the origin and such that  $\operatorname{Vol}(\O)=\operatorname{Vol}(\B)$. Let $\psi\in L^2(\O),\, \psi\geq 0$. Let $\phi\in W^{1,2}(\O)$ be the solution of 
\begin{equation}\begin{cases}
-\Delta \phi=\psi\text{ in }\O,\, 
\\ \phi=0\text{ on }\partial \O,\end{cases}
\end{equation}
and $\tilde\phi$ be the solution of 
\begin{equation}\begin{cases}
-\Delta \tilde\phi=\psi^\#\text{ in }\B, 
\\ \tilde\phi=0\text{ on }\partial \B.\end{cases}
\end{equation}
Then the inequality 
\begin{equation}
\phi^\#\leq \tilde\phi\end{equation} holds pointwise in $\B$.
\end{proposition}

% \begin{equation}\begin{cases}-\Delta w=f\text{ in }\O,\, \\ w=0\text{ on }\partial \O,
% \end{cases}\end{equation}
% and if $\overline w$ solves
% \begin{equation}\begin{cases}-\Delta \overline w=f^\#\text{ in }\O,\, \\ \overline w=0\text{ on }\partial \O,
% \end{cases}\end{equation}
% then 
% \begin{equation}w^\#\leq  \overline w.\end{equation}

%\paragraph{Technical material}
The proof of Theorem \ref{Th:ExistMu} relies on some results of Alvino, Lions, and Trombetti \cite{Alvino1989}.  These results have been  used to show existence properties for two-phases optimisation problems in the case of balls by Conca, Mahadevan, and Sanz \cite{ConcaMahadevanSanz}.  The strategy from \cite{Alvino1989}  reads as follows:  using a suitable rearrangement  one checks  that, when $\O=\mathbb B(0;R)$  for a suitable $R>0$,  one can restrict to  minimising sequences  of radially symmetric functions.  Such symmetry then enables to use  a  powerful compactness result  to obtain existence of a minimiser. What is notable in our approach is that the structure of the biharmonic operator makes it so that we do not require any symmetry property of the domain,  nor  of the elements of the minimising sequence. 

 Let us introduce a comparison relation: for any two functions $f,g\in L^2(\O),\, f,g\geq 0$, we write
\begin{equation}
f\prec g\end{equation} if, for any $r\in [0; R^\#]$, 
\begin{equation}
\int_{\mathbb B(0;r)} f^\#\leq \int_{\mathbb B(0;r)} g^\#\end{equation} and if 
\begin{equation}
\int_{ \B(0,R^\#)} f^\#=\int_{ \B(0,R^\#)} g^\#.\end{equation}

\begin{remark}\textit{ It should be noted that if $g$ is $L^\infty$, and if $f\prec g$, then $f$ is $L^\infty$ as well and $ \Vert f\Vert_{L^\infty}\leq \Vert g\Vert_{L^\infty}$}. %\begin{equation}\Vert f\Vert_{L^\infty}\leq \Vert g\Vert_{L^\infty}.\end{equation}}
\end{remark}

Let $\O^\#=\mathbb B(0; R^\#)$, and  let $\B^*:=\B(0;r^*)$ be the only ball centered in the origin of volume $({D_0-\Vol(\Omega)})/{\beta_0}$. 
We define $\overline D^\#$ as
\begin{equation}\overline D^\#=1+\beta_0\mathds 1_{\B^*}.\end{equation} First of all let us notice that for any $D\in \mathcal N(\O)$ we have 
\begin{equation}D^\#\prec \overline D^{\#} \ \ \UUU \text{and} \EEE \ \  \int_\O D=\int_{\O^\#}\overline D^\#.\end{equation} 
We define the class
\begin{equation}
\mathscr C\left(\overline D^\#\right):=\left\{f\in L^2(\O)\ : \ f\geq 0,\,  f^\#=\overline D^\#\right\},\end{equation}  which exactly corresponds to the set of bang-bang functions:
\begin{equation}\operatorname{Ext}(\mathcal N(\O))=\mathscr C\left(\overline D^\#\right).\end{equation}
This class $\mathscr C(\psi)$ is not closed under  weak-$\ast$ $L^\infty $  convergence.  Its  weak-$\ast$ $L^\infty $  compactification has been  proved  in \cite{Alvino1989}   to be 
\begin{equation}
\mathscr K\left(\overline D^\#\right):=\left\{f\in L^2(\O)\ : \ f\geq 0,\, f\prec \overline D^\#\right\}.\end{equation}
From \cite[Theorem 2.2]{Alvino1989}, $\mathscr K\left(\overline D^\#\right)$ is  closed and  weakly-$\ast$  compact for the $L^\infty $-topology (this result is a generalisation of a result by Migliaccio \cite{Migliaccio}).  Furthermore,  from \cite[Theorem 2.2]{Alvino1989} we have
\begin{equation}
\operatorname{Ext}\left(\mathscr K\left(\overline D^\#\right)\right)=\mathscr C\left(\overline D^\#\right).
\end{equation}
As a consequence of  the  general result \cite[Proposition 2.1]{HenrotPierre} or directly from  weak-$\ast$  convergence to extreme points of convex sets, if a sequence $\{f_k\}_{k\in \N}\in \mathscr K\left(\overline D^\#\right)$  weakly-$\ast$  converges to $f\in \mathscr C\left(\overline D^\#\right)$, then the convergence is strong in $L^p$, $p\in [1;+\infty)$,  see \cite{Visintin}.

\subsection{Proof of Theorem \ref{Th:ExistMu}}\label{sec:add}
What should be noted is that, here, the  weak-$\ast$ $L^\infty $  convergence of a sequence $\{D_k\}_{k\in \N}\in \mathcal N(\O)^\N$ does not imply the convergence of the associated sequence of eigenvalues $\{\mu(D_k)\}_{k\in \N}$. As is clear from the eigenequation 
\begin{equation}\label{Eq:Mu}
\begin{cases}
\Delta \left(D\Delta u_\rho\right)=\mu(D)u_D\text{ in }\O, 
\\ u_D=\Delta u_D=0\text{ on }\partial \O,\end{cases} \end{equation}
the correct convergence that would imply lower-semi continuity of the eigenfunction is the convergence of the sequence $\left\{\frac1{D_k}\right\}_{k\in \N}$. 

\begin{lemma}\label{Cl:CvInv}
 Let $\delta$ and $M_1$ be two positive constants.  Let $\{D_k\}_{k\in \N}\in L^\infty(\O)^\N\,$, $\inf_{k,\O}D_k\geq \delta>0\,$,   and   $\sup_k \Vert D_k\Vert_{L^\infty(\O)}\leq M_1$. Assume there exists $C_\infty\in L^\infty(\O)$  such that 
\begin{equation}
\frac1{D_k}\underset{k\rightarrow +\infty}\rightharpoonup C_\infty \text{  weakly-$\ast$ in $L^\infty$.  }
\end{equation}
Then, up to a subsequence, 
\begin{equation}
\mu (D_k)\underset{k\to \infty}\rightarrow \mu\left(\frac1{C_\infty}\right).
\end{equation}
\end{lemma}

\begin{proof}[Proof of  Lemma \ref{Cl:CvInv}]
To  lighten  notations, for any $k\in \N$  we denote by 
$u_k$ the eigenfunction associated with $\mu(D_k)$ and we define   \begin{equation}\label{Eq:uk}z_k = -D_k \Delta u_k.\end{equation}
 From Lemma \ref{Cl:Positivity}, we  have $z_k=0$ on $\partial \O$ and $z_k\geq 0$ in $\O$. 

Furthermore  from  Lemma  \ref{Cl:Misc} the sequence $\left\{\mu(D_k)\right\}_{k\in\N}$ is bounded.  We  can thus choose $\mu_\infty\in \R$ such that 
$\mu(D_k)\underset{k\to \infty}\rightarrow \mu_\infty$  for some not relabelled subsequence.

By assumption, we know that 
$\frac1{D_k}\underset{k\to \infty}\rightharpoonup C_\infty \text{  weakly-$\ast$ in $L^\infty$}.$  Since 
$\frac1{1+\beta_0}\leq \frac1{D_k}\leq 1$, the same $L^\infty$ bounds hold for $\frac1{C_\infty}.$
By  Lemma  \ref{Cl:Misc}, we have a uniform $W^{2,2}(\O)$ bound on the family $\{u_k\}_{k\in \N}$. Since $z_k$ solves the equation 
\begin{equation}\label{Eq:zk}
\begin{cases}
-\Delta z_k=\mu(D_k)u_k\text{ in }\O, 
\\ z_k=0\text{ on }\partial \O\end{cases}\end{equation}
we obtain a uniform $W^{1,2}_0(\O)$ bound on $\{z_k\}_{k\in \N}$,  namely,  there exists $M$ such that 
\begin{equation}
\forall k\in \N,\, \Vert z_k\Vert_{W^{1,2}_0(\O)}\leq M.\end{equation}

As a consequence, there exists $z_\infty \in L^2(\O)$ such that 
\begin{equation}z_k\underset{k\to \infty}\rightarrow z_\infty \text{ weakly in $W^{1,2}_0(\O)$  and   strongly in $L^2(\O)$}\end{equation}
 for some not relabelled subsequence.  
There also exists $u_\infty\in W^{2,2}(\O)\cap W^{1,2}_0(\O)$ such that 
\begin{equation}u_k\underset{k\to \infty}\rightarrow u_\infty \text{ weakly in $W^{2,2}(\O)$  and  strongly in $W^{1,2}_0(\O)$}\end{equation}
 for some not relabelled subsequence.  
Passing to the limit in the weak formulation of \eqref{Eq:uk}, the triple $(u_\infty,C_\infty,z_\infty)$ solves
\begin{equation}
\begin{cases}
-\Delta u_\infty=C_\infty z_\infty \text{ in }\O, 
\\z_\infty=0\text{ on }\partial \O, \end{cases}\end{equation} and since, for any $k$, $u_k\geq 0$ and $\int_\O u_k^2=1$, we have 
$$u_\infty\geq 0 \ \ \UUU \text{and} \ \ \EEE \int_\O u_\infty^2=1.$$
Passing to the limit in the weak formulation \eqref{Eq:zk} we obtain that  $(z_\infty,\mu_\infty,u_\infty)$ solves 
\begin{equation}
\begin{cases}
-\Delta z_\infty=\mu_\infty u_\infty\text{ in }\O,
\\ z_\infty=0\text{ on }\partial \O.
\end{cases}
\end{equation}
As a consequence, $(C_\infty,u_\infty)$ solves 
\begin{equation}
\begin{cases}
\Delta\left(\frac1{C_\infty}\Delta u_\infty\right)=\mu_\infty u_\infty\text{ in }\O, 
\\ u_\infty=\Delta u_\infty=0\text{ on }\partial \O, 
\\ u_\infty\geq 0,\, \int_\O u_\infty^2=1.
\end{cases}
\end{equation}
However, the first eigenvalue being the only having a constant sign eigenfunction, we conclude that $(u_\infty,\mu_\infty)$ is the first eigencouple associated to  $\frac1{C_\infty}$  or, in other words, that 
$\mu_\infty=\mu\left(\frac1{C_\infty}\right).$ Thus, the sequence $\{\mu(D_k)\}_{k\in \N}$ has a unique closure point, and hence the entire sequence converges, so that 
\begin{equation}\underset{k\to \infty}\lim \mu(D_k)=\mu\left(\frac1{C_\infty}\right).\end{equation} 
\end{proof}

We now treat the optimisation problem \eqref{Eq:PvMu} in a slightly  different  way. \UUU For \EEE any $D\in L^\infty(\O)$ \UUU with \EEE $ \inf D>0$ \UUU we \EEE  set 
\begin{equation}\eta(D):=\mu\left(\frac1D\right).\end{equation}  We recall that from  Lemma  \ref{Cl:Reformulation}  and Subsection \ref{sub:re}  we have

$$\inf_{D\in \mathcal N(\O)}\mu(D)=\inf_{\{D\in \mathcal N(\O):\, D^\#=\overline D^\#\}}\mu(D).$$ Since $ \mathscr C\left(\overline D^\#\right)=\{D\in \mathcal N(\O)\ : \  D^\#=\overline D^\#\}$ this  can be equivalently rewritten as 
\begin{equation}
\inf_{D\in \mathcal N(\O)}\mu(D)=\inf_{D\in \mathscr C\left(\overline D^\#\right)}\mu(D).
\end{equation}
 Eventually,  as $\overline D^\#$ is bang-bang,  it follows that  $D\in \mathscr C\left(\overline D^\#\right)$ if and only if $\frac1D\in \mathscr C\left(\left(\frac1{\overline D^\#}\right)^\#\right)$. 

Given the definition of $\eta$,  problem  \eqref{Eq:PvMu} is equivalent to
\begin{equation}\label{Eq:La}
\inf_{ \left\{E\in \mathscr C\left(\left(\frac1{\overline D^\#}\right)^\#\right)\right\}}\eta\left(E\right),\end{equation} in the sense that, if $E$ solves \eqref{Eq:La} then $\frac1E$ solves \eqref{Eq:PvMu}.

The key lemma is thus the following:
\begin{lemma}\label{Cl:Eta} 
\begin{enumerate}
\item The variational problem 
\begin{equation}\label{Eq:PvEta}
\inf_{E\in  \mathscr K \left(\left(\frac1{\overline D^\#}\right)^\#\right)}\eta\left(E\right)\end{equation} 
has a solution $E^*$.
\item The solutions of the variational problem \eqref{Eq:PvEta} belong to $ \mathscr C\left(\left(\frac1{\overline D^\#}\right)^\#\right)$.
\end{enumerate}
\end{lemma}

\begin{proof}[Proof of  Lemma \ref{Cl:Eta}]  \UUU Point 1. \EEE The existence of a minimiser  for problem \eqref{Eq:PvEta}  follows from the  weak-$\ast$  $L^\infty $  compactness of the set $ \mathscr K\left(\left(\frac1{\overline D^\#}\right)^\#\right)$.  Let  $\{E_k\}_{k\in \N}\in \mathscr K\left(\left(\frac1{\overline D^\#}\right)^\#\right)^\N$  be a minimising sequence, and let  $E_\infty\in \mathscr K\left(\left(\frac1{\overline D^\#}\right)^\#\right)$  be one of its weak closure points.  From  Lemma  \ref{Cl:CvInv},
\begin{equation}\eta(E_k)=\mu\left(\frac1{E_k}\right)\underset{k\to \infty}\rightarrow \mu\left(\frac1{E_\infty}\right)=\eta\left(E_\infty\right).\end{equation} Hence,  $E_\infty$  is a solution of \eqref{Eq:PvEta}.

\UUU Point 2. \EEE To prove the second point of the lemma, it suffices to prove that no interior point $E\in \mathscr K\left(\left(\frac1{\overline D^\#}\right)^\#\right)$ satisfies local first order optimality conditions. By standard theorems  \cite{Kato}  the simplicity of $\eta\left(E\right)$, obtained as in Lemma \ref{Cl:Misc}, enables  us  to differentiate it with respect to $E$.   Let $E\in \mathscr K\left(\left(\frac1{\overline D^\#}\right)^\#\right)$ and $h$ be an admissible perturbation at $E$ (i.e. $E+th\in  \mathscr K\left(\left(\frac1{\overline D^\#}\right)^\#\right)$ for $t>0$ small enough).  For the sake of notational simplicity, let  $u_E$ be the eigenfunction associated with $\eta(E)$. Let $\dot \eta$ and $\dot u$ be the derivative of $\eta(E+th)$ and its associated eigenfunction with respect to $t$ evaluated in the origin,  respectively.   Then, $(\dot u,\, \dot \eta)$ solves
\begin{equation}
\label{eq:system}
\begin{cases}
\Delta \left(\frac1E\Delta \dot u\right)-\Delta \left(\frac{h}{E^2} \Delta u_E\right)=\dot \eta u_E+\eta\left(E\right)\dot u, 
\\ \dot u=\Delta \dot u=0\text{ on }\partial \O, 
\\ \int_\O u_E\dot u=0.
\end{cases}
\end{equation}
As a consequence,  testing \eqref{eq:system} against $u_E$, using the eigenequation for $u_E$, and integrating by parts, from the fact that $\int_{\Omega} u_E^2=1$, we find  
\begin{equation}\dot \eta=\int_\O \frac{h}{E^2}\left(\Delta u_E\right)^2.\end{equation}
Thus, if $E$ is not a bang-bang function, that is, if $\omega_0:=\left\{\frac1{1+\beta_0}<E<1\right\}$ is a set of positive measure, there exists a constant $C$ such that $\frac{(\Delta u_E)^2}{E^2}=C\text{ in }\omega_0,$ 
 see for instance \cite[Theorem 1, Remark 1]{Privat2015}.  
Plugging this in the eigenequation 
$\Delta\left(\frac1E\Delta u_E\right)=\eta(E)u_E$ we obtain 
\begin{equation*}
u_E=0\text{ in }\omega_0.
\end{equation*}  This contradicts the positivity of $u_E$  inside $\O$, which is a consequence of the strong maximum principle and of Lemma \ref{Cl:Misc}. 
\end{proof}

 Relying on Lemma \ref{Cl:Eta}, we \UUU can eventually  prove \EEE Theorem \ref{Th:ExistMu} by computing  %If this lemma is satisfied, then 
\begin{align*}\inf_{D  \in  \mathscr C\left(\overline D^\#\right)}\mu(D)&=\inf_{D\in  \mathscr C\left(\overline D^\#\right)}\eta\left(\frac1D\right)=\min_{E\in  \mathscr K\left(\left(\frac1{\overline D^\#}\right)^\#\right) }\eta(E)\\
                                                                        &\quad =\min_{E\in  \mathscr C\left(\left(\frac1{\overline D^\#}\right)^\#\right) }\eta(E)=\eta(E^*)=\mu\left(\frac1{E^*}\right).\end{align*}
                                                                      Since $ E^* \in \mathscr C\left(\frac1{\overline D^\#}\right)^\#$,  we have that  $\frac{1}{ E^*}\in \mathscr C\left(\overline D^\#\right)$.  This entails the existence of a minimizer, hence Theorem \ref{Th:ExistMu} \UUU holds. \EEE

\section{Proof of Theorem \ref{Th:PV0}}
\label{sec:thm2}
 Recall that here $\Omega=\B(0,R)$ for some $R>0$. 
The core idea of the proof is to use \UUU the Talenti \EEE inequality, as was done in \cite{Anedda} to solve \eqref{Eq:PvMu}. Let us briefly recall this inequality:

Let $D\in \mathcal N(\O)$ and $u_D$ be the associated eigenfunction  solving  \eqref{Eq:MuD}. Let $z_D$ be defined as
\begin{equation}\label{Eq:Mainz}
-\Delta u_D=z_D\text{ in }\O.
\end{equation}
Since $u_D\in H^2(\O)$  we have that  $z_D\in L^2(\O)$.

 From  $\Delta u_D=0$ on $\partial \O$ we obtain $z_D=0$ on $\partial \O$. Furthermore, from  Lemma  \ref{Cl:Positivity},  there holds  $z_D \geq 0$  in $\Omega$.  Let us consider its Schwarz rearrangement $z_D^\#$.  Since $\Omega$ is a ball centered in the origin, clearly $\Omega^\#=\Omega$. Let  $\tilde u_D$ be the solution of 
\begin{equation}\label{Eq:Maintz}
-\Delta \tilde u_D=z_D^\#\text{ in }\O, 
\\ \tilde u_D=0\text{ on }\partial \O.
\end{equation}

From \UUU the Talenti \EEE inequality, Proposition \ref{Pr:Talenti}, we have 
\begin{equation}
0\leq u_D^\#\leq \tilde u_D\text{ in }\O.\end{equation} This inequality holds pointwise and hence guarantees
\begin{equation}\label{Eq:Norm}
1=\int_\O u_D^2=\int_{\O}\left(u_D^\#\right)^2\leq \int_{\O} \tilde u_D^2.
\end{equation}  Furthermore, since $z_D^\#$ is a rearrangement of $z_D$, for any $V \in [0;\Vol(\O)]$  we have that  
\begin{equation}\label{Eq:Bathtub}\inf_{F\subset \O,\, \Vol(F)=V}\int_{F} z_D^2= \inf_{G\subset \B,\, \Vol(G)=V}\int_{G}(z_D^\#)^2.\end{equation}

 Take now  $V=({D_0-\Vol{\O}})/{\beta_0}$. The function $z_D^\#$ being non-increasing, the  so-called {\it bathtub} principle  \cite[Theorem 1.14]{LiebLoss}  ensures that 
 \begin{equation}\label{Eq:Bathtub2}\inf_{G\subset \B,\, \Vol(G)=V}\int_{\O}(1+\beta_0\mathds 1_G)(z_D^\#)^2=\int_\B\overline D_\# (z_D^*)^2.
 \end{equation}
On the one hand, the Schwarz rearrangement is measure preserving, hence 
\begin{equation}
\int_\O (\Delta u_D)^2=\int_\O z_D^2=\int_\B (z_D^\#)^2=\int_\B \left(\Delta \tilde u_D\right)^2.\end{equation}
On the other hand, from \eqref{Eq:Bathtub}-\eqref{Eq:Bathtub2} we get 

\begin{equation}\label{Eq:Dr}
\int_\O D(\Delta u_D)^2\geq \int_\B \overline D_\# \left(\Delta \tilde u_\rho\right)^2.
\end{equation}
Combining \eqref{Eq:Dr} with \eqref{Eq:Norm} and plugging these estimates in the Rayleigh-quotient formulation of the eigenvalues we obtain
\begin{equation}
\mu(D)=\frac{\int_\O D(\Delta u_D)^2}{\int_\O u_D^2}\geq \frac{\int_\B \overline D_\#(\Delta \tilde u_D)^2}{\int_\O \tilde u_D^2} \geq \mu(\overline D_\#),\end{equation}
 and the assertion follows.

\section{Proof of Theorem \ref{Th:Stability}}
\label{sec:thm3} 
We are now working  under the assumption that $\Omega=\B(0,R)$ for some $R>0$. Recall  that $\rho^*=\mathds 1_{\mathbb B^*}$ is the characteristic function of a  ball centered in the origin and  of volume $V_0$. By the same arguments as in \cite[Theorem 3.3]{Anedda}, $\rho^*$ is the unique minimiser of $\lambda_0$ in $\mathcal M(\O)$ and $u_{0,\rho^*}$ is radially symmetric non-increasing. 

Furthermore, $u_{0,\rho^*}$ is strictly decreasing and there holds  \begin{equation}\label{Eq:De}\forall \e>0,\, \exists \delta(\e)>0,\, \forall r \in (\e;R]: \quad  \left|\frac{\partial u_{0,\rho^*}}{\partial r}\right|\geq \delta (\e).\end{equation}
 Indeed, this follows from the following fact: replacing $u_{0,\rho^*}$ with the solution $w$ to 
$$\begin{cases}-\Delta w=|\Delta u|^*&\text{ in }\B(0;R), 
\\ w\in W^{1,2}_0(\O),\end{cases}$$ we obtain, combining the arguments of Lemma \ref{Cl:Misc} and \UUU the Talenti \EEE inequality, that 
$$\frac{\int_\O\left(\Delta w\right)^2-\int_\O \rho^* w^2}{\int_\O w^2}\leq \frac{\int_\O\left(\Delta u_{0,\rho^*}\right)^2-\int_\O \rho^* u_{0,\rho^*}^2}{\int_\O u_{0,\rho^*}^2}.$$ Thus, $w$ is also an eigenfunction. By simplicity of $\lambda_0(\rho^*)$, $u_{0,\rho^*}$ and $w$ are linearly dependent. As a consequence,  $u_{0,\rho^*}=cw$ for some constant $c>0$; this sign condition comes from the fact that both $u_{0,\rho^*}$ and $w$ are non-negative. Thus, it follows that $-\Delta u_{0,\rho^*}=\left| \Delta u_{0,\rho^*}\right|^*$. Since $\Delta u_{0,\rho^*}\neq 0$, $\left| \Delta u_{0,\rho^*}\right|^*(0)>0$. Setting $z=\Delta u_{0,\rho^*}=-|\Delta u_{0,\rho^*}|^*$ we have, in radial coordinates
$$r\frac{\partial u_{0,\rho^*}}{\partial r}(r)=\int_0^r\tau z(\tau)d\tau<0,$$ which concludes the proof.

 We   prove Theorem \ref{Th:Stability}  by contradiction and assume that, for any $\alpha>0$, there exists a radially symmetric  $\rho_\alpha \in \mathcal M(\O)$ such that 
\begin{equation}
\label{eq:contr}
\lambda_\alpha(\rho_\alpha)\leq \lambda_\alpha(\rho^*), \  \rho_\alpha \neq \rho^*.\end{equation}
 Let us prepare a preliminary lemma. 

\begin{lemma}\label{Cl:Convergence}
We can assume that $\rho_\alpha$ is a bang-bang function. Furthermore, we have $\rho_\alpha\underset{\alpha \to 0}\rightarrow \rho^*$ strongly in $L^1$.
\end{lemma}
\begin{proof}[Proof of  Lemma \ref{Cl:Convergence}]
The first point follows from the concavity of the functional. For the second point, we first observe that, for any  weak-$\ast$ $L^\infty $  closure point $\rho_0$ of $\{\rho_\alpha\}_{\alpha\to 0}$, there holds $\lambda_0(\rho_0)\leq \underset{\alpha\to 0}{\lim\inf}\lambda_\alpha(\rho_\alpha)$: setting, for notational convenience, $u_\alpha:=u_{\alpha,\rho_\alpha}$, we have, from Lemma \ref{Cl:Misc}, a uniform $W^{2,2}(\O)$ bound on this sequence. This allows to pick a weak $W^{2,2}(\O)$ and strong $L^2(\O)$ closure point $u_0$ of a not relabelled subsequence. 
By  weak lower-semicontinuity of convex functions
$$\int_\O (\Delta u_0)^2\leq \underset{\alpha \to 0}{\lim \inf}\int_\O (1+\alpha \rho_\alpha)(\Delta u_\alpha)^2,$$ while 
$$\int_\O \rho_0u_0^2=\underset{\alpha \to 0}\lim \int_\O \rho_\alpha u_\alpha^2,\, \int_\O u_0^2=1.$$ From the variational formulation \eqref{Eq:DefLambda} of $\lambda_0(\rho_0)$, we obtain the conclusion.  Let us then observe that for any fixed $\rho\in \mathcal M(\O)$ (in particular for $\rho=\rho^*$), there holds $\lambda_0(\rho)=\underset{\alpha \to 0}\lim \lambda_\alpha(\rho)$. Passing to the limit in the inequality $\lambda_\alpha(\rho_\alpha)\leq \lambda_\alpha(\rho^*)$ we obtain $\lambda_0(\rho_0)\leq \lambda_0(\rho^*)$. Since $\rho^*$ is the unique minimiser of $\lambda_0$ we have $\rho_0=\rho^*$. As $\rho^*$ is an extreme point of $\mathcal M(\O)$, from \cite[Proposition 2.1]{HenrotPierre} this convergence is strong in $L^1$.
\end{proof}  

Henceforth,  we can hence assume that the  sequence $\{\rho_\alpha\}_{\alpha \to 0}$  fulfilling \eqref{eq:contr} consists of bang-bang functions. 
We use this information to proceed with the proof, which rests upon fine properties of the switch function. We  need to use one of the core idea of $H$-convergence to make sure this function is regular enough. Let us explain why some concepts from homogenisation are needed: if we consider the map $D\mapsto \lambda_\alpha(D)$  and   if we define $u_{\alpha,\rho}$ as the eigenfunction associated with $\lambda_\alpha(\rho)$, the simplicity of the eigenvalue ( Lemma  \ref{Cl:Misc}) ensures that $\rho\mapsto (\lambda_\alpha(\rho),\, u_{\alpha,\rho})$ is G\^ateaux-differentiable. Furthermore, for any $\rho \in \mathcal M(\B(0,R))$ and any admissible perturbation $h$ at $\rho$ (i.e a function $h$ such that, for every $\e>0$ small enough $\rho +\e h\in \mathcal M(\B(0,R))$), the G\^ateaux-derivatives $\dot u_{\alpha,\rho}$ and $\dot\lambda_\alpha(\rho)$ (we omit the dependency on $h$ for notational convenience) solve
\begin{equation}\label{Eq:GateauxDerivative}
\begin{cases}
\Delta \left((1+\alpha \rho)\Delta \dot u_{\alpha,\rho}\right)+\alpha\Delta\left(h\Delta u_{\alpha,\rho}\right)=\left(\lambda_{\alpha,\rho}{+}\rho\right)\dot u_{\alpha \rho}+\left(\dot \lambda_{\alpha,\rho}{+}h\right)\uar \text{ in }\B(0,R), 
\\ \duar=\Delta \duar=0\text{ on }\partial \B(0,R),
\\ \int_{\B(0,R)} \uar \duar=0.
\end{cases}
\end{equation}
Multiplying the equation by $\uar$, integrating by parts, and using the equation \eqref{Eq:LambdaRho} on $\uar$  we obtain the following expression for $\dot\lambda_\alpha(\rho)$:
\begin{equation}\label{Eq:LambdaDot}
\dot \lambda_\alpha(\rho)=\int_{\B(0,R)} h\left\{\alpha\left(\Delta \uar\right)^2-\uar^2\right\}.
\end{equation}
This leads to defining the switch function associated with  the  problem  as 
\begin{equation}\label{Eq:SwitchDeb}U_{\alpha,\rho}:= \alpha\left(\Delta \uar\right)^2-\uar^2   .\end{equation} In other words, with this approach, we have 
$\dot \lambda_\alpha(\rho)=\int_{\B(0,R)} U_{\alpha,\rho} h.$ Ideally, we would use Lemma \ref{Cl:Convergence} to approximate $U_{\alpha,\rho}$ by $U_{0,\rho^*}$ in the $\mathscr C^1$ norm. However, since $\Delta \uar$ is merely $L^\infty$, $U_{\alpha,\rho}$ is not regular enough. 
To overcome this problem, we  rely on some general ideas borrowed from  $H$-convergence and homogenisation theory \cite{Allaire,MuratTartar}. We  introduce, for any $\rho \in \mathcal M(\B(0,R))$, the harmonic mean $\mathscr J_-(\rho)$ of $1+\alpha \rho$,  defined as  
\begin{equation}\mathscr J_-(\rho):=\frac{1+\alpha}{1+\alpha(1-\rho)}.\end{equation}
We define an auxiliary eigenvalue $\Lambda_\alpha(\rho)$ as follows:
\begin{equation}
\Lambda_\alpha(\rho):=\min_{u\in W^{2,2}(\B(0,R))\cap W^{1,2}_0(\B(0,R))\, u\neq 0}\frac{\int_{\B(0,R)} \mathscr J_-(\rho)(\Delta u)^2-\int_{\B(0,R)} \rho u^2}{\int_{\B(0,R)} u^2}.
\end{equation} From this variational formulation, since $\rho\mapsto \mathscr J_-(\rho)$ is concave,  we have that   $\rho\mapsto \Lambda_\alpha(\rho)$ is concave  too.  If $\rho$ is a bang-bang function, that is, if $\rho=\mathds 1_E$ for some measurable subset $E$, then 
$\mathscr J_-(\rho)=1+\alpha \rho$ so that 
\begin{equation}\text{For any bang-bang $\rho$  one has that} \ \lambda_\alpha(\rho)=\Lambda_\alpha(\rho).
\end{equation}

Hence  for all $\alpha>0, $  we have that
$$\lambda_\alpha(\rho_\alpha)=\Lambda_\alpha(\rho_\alpha) \ \text{and} \  \lambda_\alpha(\rho^*)=\Lambda_\alpha(\rho^*).
$$
To see why this allows to overcome the aforementioned regularity issues, let us compute the G\^ateaux-derivative of the map $\rho \mapsto \Lambda_\alpha(\rho)$. Let us define $v_{\alpha,\rho}$  to be   the eigenfunction associated with $\Lambda_\alpha(\rho)$.  This  can be chosen positive  and normalized in $L^2$.  In particular, $v_{\alpha,\rho}$ solves
\begin{equation}\label{Eq:Valpha}\begin{cases}
\Delta\left(\mathscr J_-(\rho)\Delta v_{\alpha,\rho}\right)=\Lambda_\alpha(\rho)v_{\alpha,\rho}+\rho v_{\alpha,\rho}\text{ in }\B(0,R), 
\\v_{\alpha,\rho}=�\Delta v_{\alpha,\rho}=0\text{ on }\partial \B(0,R), 
\\ \int_{\B(0,R)} v_{\alpha,\rho}^2=1,\, v_{\alpha,\rho}\geq 0.\end{cases}\end{equation}
From the same arguments as in  Lemma  \ref{Cl:Misc}, $\Lambda_\alpha(\rho)$ is a simple eigenvalue, and so the map $\rho \mapsto \left(\Lambda_\alpha(\rho),\, v_{\alpha,\rho}\right)$ is G\^ateaux-differentiable and, for $\rho \in \mathcal M(\B(0,R))$ and an admissible perturbation $h$ at $\rho$, if we denote with a dot the G\^ateaux-differentiated quantities, the couple $\left(\dot \Lambda_\alpha(\rho),\, \dot v_{\alpha,\rho}\right)$ solves 
\begin{equation}\label{Eq:ValphaDot}
\begin{cases}
\Delta \left(\mathscr J_-(\rho) \Delta \dot v_{\alpha,\rho}\right)+\frac\alpha{1+\alpha} \Delta\left(h \mathscr J_-(\rho)^2 \Delta  v_{\alpha,\rho} \right)=&\left(\Lambda_\alpha(\rho)+\rho\right)\dot  v_{\alpha,\rho}\\&+\dot \Lambda_\alpha(\rho)  v_{\alpha,\rho}+h v_{\alpha,\rho}\text{ in }\B(0,R), 
\\� \dot v_{\alpha,\rho}=\Delta  \dot v_{\alpha,\rho}=0\text{ on }\partial \B(0,R), 
\\�\int_{\B(0,R)}  \dot v_{\alpha,\rho} v_{\alpha,\rho}=0.
\end{cases}
\end{equation} This equation has a unique solution by the Fredholm alternative.
Multiplying  the first equation in  \eqref{Eq:ValphaDot} by $ v_{\alpha,\rho}$, integrating by part and using \eqref{Eq:Valpha} yields 
\begin{equation}
\dot \Lambda_\alpha(\rho)=\int_{\B(0,R)} h\left\{\frac\alpha{1+\alpha}\mathscr J_-(\rho)^2 (\Delta  v_{\alpha,\rho})^2- v_{\alpha,\rho}^2\right\}.\end{equation}
The new switch function 
\begin{equation}
\psi_{\alpha,\rho}:= \frac\alpha{1+\alpha}\mathscr J_-(\rho)^2 (\Delta  v_{\alpha,\rho})^2- v_{\alpha,\rho}^2 \end{equation} is now more regular, since the function $\mathscr J_-(\rho)\Delta v_{\alpha,\rho}$ is itself the solution of  an elliptic problem.  Let us now consider the two bang-bang-densities $\rho_\alpha,\, \rho^* \in \mathcal M(\B(0,R))$. Instead of considering the path $t\mapsto \lambda_\alpha(\rho_\alpha+t(\rho^*-\rho_\alpha))$, which would lead to the irregular switch function  \eqref{Eq:SwitchDeb}, we  set $\rho_t:=\rho_\alpha+t(\rho^*-\rho_\alpha)$ and we  consider the path 
\begin{equation}
f_\alpha:t\mapsto  \Lambda_\alpha( \rho_t).\end{equation} 
 For $t\in [0;1]$, let us define 
 $v_t$  to be  the eigenfunction associated with $\Lambda_\alpha(\rho^*+t(\rho_\alpha-\rho^*))$ and 
\begin{equation}\Psi_t:=\frac\alpha{1+\alpha}\mathscr J_-(\rho_t)^2 (\Delta  v_{t})^2- v_{t}^2\end{equation} By  Lemma \ref{Cl:Convergence} and by  the mean value  Theorem, we  write
\begin{align}
  \lambda_\alpha(\rho^*)-\lambda_\alpha(\rho_\alpha)&=\Lambda_\alpha(\rho^*)-\Lambda_\alpha(\rho_\alpha)=\int_{\B(0,R)} \Psi_t (\rho^*-\rho_\alpha)\end{align}
for some $t=t(\alpha)\in [0;1]$. %We fix such a $t(\alpha)$.
From Lemma \ref{Cl:Convergence},  we know that  $\rho_{t(\alpha)}\underset{\alpha \to 0}\rightarrow \rho^*$ strongly in $L^1(\B(0,R))$. From standard elliptic regularity,  there exists a constant $M>0$ such that $\Vert\mathscr J_-(\rho_{t(\alpha)})\Delta v_{t(\alpha)}\Vert_{\mathscr \UUU C^1(\B(0,R))}\leq M$.   Again  from elliptic regularity,  we also have that  $v_{t(\alpha)}\underset{\alpha\to 0}\rightarrow u_{0,\rho^*}$ in $\mathscr C^1$. Hence, $\Psi_{t(\alpha)}\underset{\alpha \to 0}\rightarrow -u_{0,\rho^*}^2$ in $\mathscr C^1$.  Since $\Psi_{t(\alpha)}$ is radial, the strict monotonicity \eqref{Eq:De}  implies that  $\rho^*$ is  the unique solution of 
\begin{equation}\label{Eq:TT}\inf_{\rho\in \mathcal M(\B(0,R))}\int_{\B(0,R)} \Psi_{t(\alpha)}\rho\end{equation}  for $\alpha>0$ small enough. Indeed, from \eqref{Eq:De} and the $\mathscr C^1$ convergence of $\left\{\Psi_{t(\alpha)}\right\}_{\alpha \to 0}$ to $-u_{0,\rho^*}^2$, for $\alpha>0$ small enough, $\mathbb B^*$ is the unique level set of $\Psi_{t(\alpha)}$ of volume $\rho_0$.
 
Hence, $\int_{\B(0,R)} \Psi_{t(\alpha)}(\rho_\alpha-\rho^*)\geq 0$, which in turn implies \UUU that \EEE $\Lambda_\alpha(\rho^*)-\Lambda_\alpha(\rho_\alpha)\leq0$.  By Lemma \ref{Cl:Convergence}, this leads to  \UUU contradicting \EEE \eqref{eq:contr} \UUU and \EEE concludes the proof of the theorem.

\section{Conclusion}
\label{sec:con}
In this article, we have studied several theoretical aspects  related with   the spectral optimisation of inhomogeneous plates. It is worth underlining that the existence result, Theorem \ref{Th:ExistMu},  is in sharp contrast with other results  in the context of  the optimisation of two-phase problems.

 Note moreover that  the stationarity of minimisers of $\lambda_\alpha$ , as $\alpha\to 0^+$  is proved with respect to  radial  competitors  only. We believe that the case of not radially symmetric competitors is presently out of reach, given the available rearrangement tools. In fact,  Theorem \ref{Th:PV0} indicates that the correct rearrangement when handling thickness optimisation  is expected to be  the {\it increasing} rearrangement,  whereas  previous results \cite{Anedda} point to the fact that optimisation with respect to the density  should rather involve  the {\it decreasing} rearrangement.

\section*{Acknowledgments}
E.~Davoli has been partially supported by the Austrian Science
  Funds (FWF) grants V662, I4052, Y1292, and F65, as well as by the
  OeAD-WTZ project CZ04/2019. I.~Mazari acknowledges support of the
 FWF grants I4052 and F65. U. Stefanelli
  acknowledges support of the FWF grants
  I4354, F65, I5149, and P\,32788, and by the OeAD-WTZ
project CZ 01/2021.

\end{document}